\documentclass[12pt,reqno]{article}

\usepackage[usenames]{color}
\usepackage{amssymb}
\usepackage{amsmath}
\usepackage{amsthm}
\usepackage{amsfonts}
\usepackage{amscd}
\usepackage{graphicx}
\usepackage{mathtools}

\mathtoolsset{showonlyrefs}

\usepackage[colorlinks=true,
linkcolor=webgreen,
filecolor=webbrown,
citecolor=webgreen]{hyperref}

\definecolor{webgreen}{rgb}{0,.5,0}
\definecolor{webbrown}{rgb}{.6,0,0}

\usepackage{color}
\usepackage{fullpage}
\usepackage{float}

\usepackage{graphics}
\usepackage{latexsym}

\begin{document}

\theoremstyle{plain}
\newtheorem{theorem}{Theorem}
\newtheorem{corollary}[theorem]{Corollary}
\newtheorem{proposition}{Proposition}
\newtheorem{lemma}{Lemma}
\newtheorem{example}{Example}
\newtheorem{remark}{Remark}

% \numberwithin{equation}{section}
% \numberwithin{example}{section}
% \numberwithin{theorem}{section}
% \numberwithin{remark}{section}
% \numberwithin{proposition}{section}
% \numberwithin{lemma}{section}

\newcommand{\braces}{\genfrac{\lbrace}{\rbrace}{0pt}{}}

\begin{center}

\vskip 1cm{\large\bf  
New Generalizations of Two Ramanujan Series for $1/\pi$}
\vskip 1cm
\large
Kunle Adegoke \\ 
Department of Physics and Engineering Physics\\
Obafemi Awolowo University\\
220005 Ile-Ife \\
Nigeria \\
\href{mailto:adegoke00@gmail.com }{\tt adegoke00@gmail.com}\\

\end{center}

\vskip .2 in

\begin{abstract}
By utilizing two known hypergeometric summation identities, we extend two \mbox{well-known} Ramanujan series for $1/\pi$ by making each of them a member of an infinite family of series. We also find the corresponding families involving harmonic numbers and odd harmonic numbers. To further illustrate our method we derive families of \mbox{Ramanujan-like} series for $1/\pi$ associated with a hypergeometric series derived by Lavoie, two hypergeometric series from Bailey's book and a recent hypergeometric series found by Campbell. Thus, by leveraging known hypergeometric summation identities, our approach bypasses the more tedious, traditional limits or complex Fourier-Legendre expansions; offering a cleaner, unified template for generating Ramanujan-like series for $1/\pi$.
\end{abstract}

\noindent 2020 {\it Mathematics Subject Classification}: Primary 33C20; Secondary 11Y60. 

\noindent \emph{Keywords:} Ramanujan series, Ramanujan-like series, harmonic number, odd harmonic number, hypergeometric series.

\section{Introduction}
Let $\left( x \right)_n  = \prod_{k = 0}^{n - 1} {\left( {x + k} \right)} $, $x$ a complex number and $n$ a non-negative integer. Our main purpose in this article is to derive the following families of Ramanujan-like series for $1/\pi$:
\begin{align}\label{gen1}
&\sum_{k = 0}^\infty  {\frac{{( - 1)^k }}{{k!^3 }}\frac{{\left( {\frac{1}{2}} \right)_k^3 \left( {4k + 1 - 2\ell} \right)}}{{\prod_{j = 1}^\ell {\left( {2k - 2j + 1} \right)} \prod_{j = 1}^m {\left( {2k - 2j + 1} \right)} \prod_{j = 1}^n {\left( {2k - 2j + 1} \right)} \prod_{j = 1}^{m - \ell} {\left( {k + j} \right)} \prod_{j = 1}^{n - \ell} {\left( {k + j} \right)} }}}\nonumber\\
&\qquad  = \frac{{( - 1)^{m + n} }}{{\left( {\frac{1}{2}} \right)_m \left( {\frac{1}{2}} \right)_n \left( {\frac{1}{2}} \right)_{m + n - \ell} 2^{\ell + m + n - 1} }}\,\frac{1}{\pi }
\end{align}
and
\begin{equation}\label{gen2}
\sum_{k = 0}^\infty  {\frac{{\left( {\frac{1}{2}} \right)_k^3 }}{{4^k k!^3 }}\frac{{\prod_{j = 1}^{2n} {\left( {2k + j} \right)} }}{{\prod_{j = 1}^n {\left( {k + j} \right)^2 \left( {2k - 2j + 1} \right)} }}\left( {6k + 1 + 2n} \right)}  = \frac{{( - 1)^n\,2^{n+2} }}{{\left( {\frac{1}{2}} \right)_n }}\, \frac1{\pi };
\end{equation}
valid for $\ell$, $m$, and $n$ non-negative integers.

Identities~\eqref{gen1} and~\eqref{gen2} are respective generalizations of the Ramanujan series~\cite{ramanujan14}:
\begin{equation}\label{bauer}
\sum_{k = 0}^\infty  {( - 1)^k \frac{{\left( \frac12 \right)_k^3 }}{{k!^3 }}\left( {4k + 1} \right)}  = \frac{2}{\pi }
\end{equation}
and
\begin{equation}
\sum_{k = 0}^\infty  {\frac{{\left( {\frac{1}{2}} \right)_k^3 }}{{4^k k!^3 }}\left( {6k + 1} \right)}  = \frac{4}{\pi }.
\end{equation}

When $\ell=0=m$, identity~\eqref{gen1} reduces to
\begin{equation*}
\sum_{k = 0}^\infty  {( - 1)^k \frac{{\left( {\frac{1}{2}} \right)_k^3 }}{{k!^3 }}\frac{{\left( {4k + 1} \right)}}{{\prod_{j = 1}^n {\left( {2k - 2j + 1} \right)\left( {k + j} \right)} }}}  = \frac{{( - 1)^n }}{{\left( {\frac{1}{2}} \right)_n^2 2^{n - 1} }}\,\frac{1}{\pi },
\end{equation*}
which was derived by Levrie~\cite[Theorem 7]{levrie10} using Fourier-Legendre theory. In fact we will see that identities~\eqref{gen1} and~\eqref{gen2} are particular cases of more general results.

Let $H_j$ be the $j^{\rm th}$ harmonic number and let $O_j$ be the $j^{\rm th}$ odd harmonic number. We will also derive the following series involving harmonic and odd harmonic numbers:
\begin{equation}\label{di4jyj5}
\sum_{k = 0}^\infty  {\frac{{( - 1)^k \left( {\frac{1}{2}} \right)_k^3 \left( {4k + 1} \right)\left( {2O_{k - n}  + H_{k + n} } \right)}}{{k!^3 \prod_{j = 1}^n {\left( {k + j} \right)\left( {2k - 2j + 1} \right)} }}}  = \frac{{( - 1)^n }}{{2^{n - 2} }}\frac{1}{{\left( {\frac{1}{2}} \right)_n^2 }}\left( {2O_n  - \ln 2} \right)\,\frac{1}{\pi }
\end{equation}
and
\begin{align}\label{z75f5e1}
&\sum_{k = 0}^\infty  {\frac{{\left( {\frac{1}{2}} \right)_k^3 }}{{4^k k!^3 }}\frac{{\prod_{j = 1}^{2n} {\left( {2k + j} \right)} }}{{\prod_{j = 1}^n {\left( {k + j} \right)^2 \left( {2k - 2j + 1} \right)} }}\left( {\left( {6k + 1 + 2n} \right)\left( {2O_{k - n}  - 2O_{k + n}  + H_{k + n} } \right) - 2} \right)}\nonumber\\
&\qquad  = ( - 1)^n\,\frac{{2^{n + 3} }}{{\left( {\frac{1}{2}} \right)_n }}\left( {O_n  - \ln 2} \right)\frac{1}{\pi }.
\end{align}

Identity~\eqref{di4jyj5} at $n=0$ reduces to the harmonic number complement of~\eqref{bauer}:
\begin{equation}\label{rt1nmvm}
\sum_{k = 0}^\infty  {( - 1)^k \frac{{\left( {\frac{1}{2}} \right)_k^3 }}{{k!^3 }}\left( {4k + 1} \right)H_{2k} }  =  - \frac{2\ln 2}{\pi },
\end{equation}
derived by Hou et al.~\cite{hou24} and also Campbell~\cite{campbell26}.

The special case $n=0$ of identity~\eqref{z75f5e1}, namely,
\begin{equation}
\sum_{k = 0}^\infty  {\frac{{\left( {\frac{1}{2}} \right)_k^3 }}{{4^k k!^3 }}\left( {\left( {6k + 1} \right)H_k  - 2} \right)}  =  - \frac{{8\ln 2}}{\pi },
\end{equation}
was derived by Hou et al.~\cite{hou24}.

The now famous identity~\eqref{bauer} was discovered in 1859 by Bauer~\cite{bauer} via a Fourier-Legendre expansion, and rediscovered in 1914 by Ramanujan who included it in his famous first letter to G.~H.~Hardy. Early proofs relied on classical techniques like swapping limit operations while later approaches utilized advanced tools like Dougall's bilateral hypergeometric theorem. For a survey of the \mbox{Bauer-Ramanujan} series~\eqref{bauer} and related results and analyses, the reader is referred to the recent article by Campbell and Levrie~\cite{campbell24}.

To provide additional examples of how our method can be applied to discover \mbox{Ramanujan-like} series associated with well-poised hypergeometric series, we will derive more families of series for~$1/\pi$ including the following:
\begin{align}\label{zm8azjv}
&\sum_{k = 0}^\infty  {\frac{{\left( {\frac{1}{2}} \right)_k^2 }}{{k!^2 }}\,\frac{{\left( {1 - 2\ell + 4k} \right)}}{{\prod_{j = 1}^\ell {\left( {2k - 2j + 1} \right)} \prod_{j = 1}^m \left({2k - 2j + 1}\right) \prod_{j = 1}^n {\left( {k + j} \right)} }}}\nonumber\\
&\qquad  = \frac{{( - 1)^{\ell + m} }}{{2^{\ell + m} }}\,\frac{{\left( {\ell + m + n - 2} \right)!\left( {\ell - m + n} \right)\left(1-2\ell\right)}}{{\left( {\frac{1}{2}} \right)_\ell \,\left( {\frac{1}{2}} \right)_m \left( {\frac{1}{2}} \right)_{\ell + n} \left( {\frac{1}{2}} \right)_{m + n} }}\,\frac{1}{\pi }
\end{align}
and
\begin{align}\label{msout8g}
&\sum_{k = 0}^\infty  {\frac{{\left( {\frac{1}{2}} \right)_k^2 }}{{k!^2 }}\,\frac{{\binom{{2k + 2n}}{{2n}}}}{{\binom{{k + n}}{n}^2 }}\,\frac{1}{{2k + n}}}\nonumber\\
&\qquad  = \frac{1}{{2n}}\frac{{n!}}{{\left( {\frac{1}{2}} \right)_n }}\left( { - \frac{8}{{\left( {2n - 1} \right)\left( {2n - 3} \right)}} + \left( {( - 1)^n  + 1} \right)\pi  + 4\sum_{k = 1}^{n - 2} {\frac{1}{{4k + 3 - 2n}}} } \right)\,\frac1\pi.
\end{align}

Identity~\eqref{zm8azjv} holds for non-negative integers $\ell$, $m$, and $n$ such that \mbox{$\ell+m+n-2$} is not a negative integer while~\eqref{msout8g} is valid for all positive integers $n$. Examples from~\eqref{msout8g} include
\begin{equation}
\sum_{k = 0}^\infty  {\frac{{\binom{{2k}}{k}C_k }}{{2^{4k} }}}=\frac4\pi
\end{equation}
and
\begin{equation}
\sum_{k = 0}^\infty  {\frac{{C_k^2 }}{{2^{4k} }}\,\frac{{\left( {2k + 1} \right)\left( {2k + 3} \right)}}{{\left( {k + 2} \right)}}}  =  - \frac{{16}}{{3\pi }} + 4.
\end{equation}

Here $C_j$ is the $j^{th}$ Catalan number,
\begin{equation*}
C_j=\frac{\binom{2j}j}{j+1}.
\end{equation*}

We will obtain another generalization of~\eqref{rt1nmvm}, somewhat simpler than~\eqref{di4jyj5}, namely,
\begin{equation*}
\sum_{k = 0}^\infty  {( - 1)^k \frac{{\left( {\frac{1}{2}} \right)_k^3 }}{{k!^3 }}\,\frac{{\left( {4k + 1} \right)H_{2k} }}{{\prod_{j = 1}^n {\left( {2k - 2j + 1} \right)\left( {k + j} \right)} }}}  = \frac{{( - 1)^n }}{{2^{n - 1} }}\,\frac{{\left( {O_n  - \ln 2} \right)}}{{\left( {\frac{1}{2}} \right)_n^2 }}\,\frac{1}{\pi }.
\end{equation*}

Let $\Gamma(z)$ be Euler's Gamma function. Generalized binomial coefficients are defined for complex numbers $r$ and $s$ by
\begin{equation}\label{y89d722}
\binom rs= \frac{{\Gamma (r + 1)}}{{\Gamma (s + 1)\Gamma (r - s + 1)}}.
\end{equation}
The {\it extended} Pochhammer relation $(x)_y$ is defined for complex numbers $x$ and $y$ that are not negative integers by
\begin{equation}\label{ex_pochhammer}
\left( x \right)_y  = \frac{{\Gamma \left( {x + y} \right)}}{{\Gamma \left( x \right)}} = \binom{{x + y - 1}}{y}\,\Gamma \left( {y + 1} \right).
\end{equation}
If, in particular,  $y=n$ is a non-negative integer, then we have the Pochhammer symbol
\begin{equation}
\left( x \right)_n  = \prod_{k = 0}^{n - 1} {\left(x + k\right)}  = \binom{{x + n - 1}}{n}\,n!
\end{equation}

Harmonic numbers, $H_j$, and odd harmonic numbers, $O_j$, are defined for non-negative integers $j$ by
\begin{equation}
H_j=\sum_{k=1}^j\frac1k,\quad O_j=\sum_{k=1}^j\frac1{2k-1},\quad H_0=0=O_0.
\end{equation}
The identity
\begin{equation}\label{nerd5x1}
H_{n-1/2}=2O_n-2\ln 2,
\end{equation}
extends harmonic numbers to half-integer arguments and is a consequence of the link between harmonic numbers, odd harmonic numbers and the digamma function.

\section{Required results}
\begin{lemma}
If $k$ and $x$ are complex numbers, then
\begin{equation}\label{yahnm05}
\binom{{k}}{{x + 1/2}} = \frac{2}{\pi }\,\frac{1}{{\left( {2x + 1} \right)}}\frac{{x!k!}}{{\left( {\frac{1}{2}} \right)_x \left( {\frac{1}{2}} \right)_k }}\binom{{k - 1/2}}{x}.
\end{equation}
\end{lemma}
\begin{proof}
The identity
\begin{equation}\label{w7han3w}
\binom{{a - b}}{c} = \binom{{a - c}}{b}\binom{{a}}{c}\binom{{a}}{b}^{ - 1} ,
\end{equation}
with $a=k$, $b=1/2$, and $c=k-1/2-x$ gives
\begin{equation}
\binom{{k - 1/2}}{x} = \binom{{x + 1/2}}{{1/2}}\binom{{k}}{{x + 1/2}}\binom{{k}}{{1/2}}^{ - 1}; 
\end{equation}
and hence~\eqref{yahnm05} since
\begin{equation}
\binom{{x + 1/2}}{{1/2}} = \frac{{2x + 1}}{{x!}}\left( {\frac{1}{2}} \right)_x 
\end{equation}
and
\begin{equation}
\binom{{k}}{{1/2}} = \frac{2}{\pi }\,\frac{{k!}}{{\left( {\frac{1}{2}} \right)_k }}.
\end{equation}
\end{proof}
\begin{lemma}\label{w00yljy}
If $x$ is a complex number that is not a negative integer, then
\begin{align}
\left( { - x + \frac{1}{2}} \right)_k  &= \frac{{\left( {\frac{1}{2}} \right)_x \left( {\frac{1}{2}} \right)_k }}{{x!}}\,\binom{{k - 1/2}}{x}^{ - 1} \cos \left( \pi x \right)\label{l5fmdri},\\
\left( {x + \frac{1}{2}} \right)_k  &= \left( {\frac{1}{2}} \right)_k \binom{{2k + 2x}}{{2x}}\binom{{k + x}}{x}^{ - 1}\label{bjv8z0f},
\end{align}
and
\begin{equation}
\left( {x - \frac{1}{2}} \right)_k  =\frac{2x-1}{\left(2k+2x-1\right)} \,\left( {\frac{1}{2}} \right)_k \binom{{2k + 2x}}{{2x}}\binom{{k + x}}{x}^{ - 1}\label{kebxljz}.
\end{equation}
In particular, if $n$ is a non-negative integer, then
\begin{align}
\left( { - n + \frac{1}{2}} \right)_k  &= ( - 1)^n \frac{{2^n \left( {\frac{1}{2}} \right)_n \left( {\frac{1}{2}} \right)_k }}{{\prod_{j = 1}^n {\left( {2k - 2j + 1} \right)} }}\label{mq7ny8m},\\
\left( {n + \frac{1}{2}} \right)_k  &= \frac{1}{{2^{2n} }}\frac{{\left( {\frac{1}{2}} \right)_k }}{{\left( {\frac{1}{2}} \right)_n }}\frac{{\prod_{j = 1}^{2n} {\left( {2k + j} \right)} }}{{\prod_{j = 1}^n {\left( {k + j} \right)} }}\label{t0khqm8},
\end{align}
and
\begin{equation}
\left( {n - \frac{1}{2}} \right)_k  = \frac{2n-1}{2k+2n-1}\,\frac{1}{{2^{2n} }}\frac{{\left( {\frac{1}{2}} \right)_k }}{{\left( {\frac{1}{2}} \right)_n }}\frac{{\prod_{j = 1}^{2n} {\left( {2k + j} \right)} }}{{\prod_{j = 1}^n {\left( {k + j} \right)} }}.
\end{equation}
\end{lemma}
\begin{proof}
To prove~\eqref{l5fmdri},
\begin{equation}
\left( { - x + \frac{1}{2}} \right)_k  = k!\,\binom{{k - x - 1/2}}{k}.
\end{equation}
Now, using~\eqref{w7han3w} with $a=k$, $b=x+1/2$, and $c=k$, write
\begin{equation}
\binom{{k - x - 1/2}}{k} = \binom{{0}}{{x + 1/2}}\binom{{k}}{k}\binom{{k}}{{x + 1/2}}^{ - 1}  = \frac{2}{\pi }\,\frac{{\cos \left( \pi x\right)}}{{2x + 1}}\binom{{k}}{{x + 1/2}}^{ - 1},
\end{equation}
since
\begin{equation}
\binom{{0}}{s}
=\begin{cases}
 \dfrac{{\sin \left( {\pi s} \right)}}{{\pi s}},&\text{if $s\ne0$;} \\ 
 1,&\text{if $s=0$.} \\ 
 \end{cases}
\end{equation}
Identity~\eqref{l5fmdri} now follows by the use of~\eqref{yahnm05}.

Consider the identity
\begin{equation}\label{zxsidip}
\binom{{a + b - 1/2}}{c} = 2^{ - 2c} \,\binom{{2c}}{c}\,\binom{{a + b}}{c}^{ - 1} \,\binom{{2\left( {a + b} \right)}}{2c},
\end{equation}
which holds for complex numbers $a$, $b$, and $c$ such that $c$ is not a negative integer.

Since
\begin{equation*}
\left( {x + \frac{1}{2}} \right)_k  = k!\binom{{k + x - 1/2}}{k},
\end{equation*}
substituting $a=k$, $b=x$, and $c=k$ in~\eqref{zxsidip} gives~\eqref{bjv8z0f}.

Similarly,
\begin{equation*}
\left( {x - \frac{1}{2}} \right)_k  = k!\binom{{k + x - 3/2}}{k}=\frac{2x-1}{2x+2k-1}\,k!\binom{{k + x - 1/2}}{k},
\end{equation*}
and hence~\eqref{kebxljz}.
\end{proof}
\begin{remark}
Since
\begin{equation}\label{wwymz81}
\Gamma \left( {\frac{1}{2} - u} \right) = \frac{{\sqrt \pi}}{{\left( {\frac{1}{2}} \right)_u\,\cos(\pi u) }},
\end{equation}
\begin{equation}\label{zvg11tu}
\Gamma \left( {\frac{1}{2} + u} \right) = \left( {\frac{1}{2}} \right)_u \sqrt \pi ,
\end{equation}
and
\begin{equation}
\Gamma \left( { - \frac{1}{2} + u} \right) = \frac{{2\sqrt \pi  }}{{2u - 1}}\left( {\frac{1}{2}} \right)_u ,
\end{equation}
we also have, using the definition~\eqref{ex_pochhammer},
\begin{align}
\left( { - x + \frac{1}{2}} \right)_k  &= \left( {\frac{1}{2}} \right)_{k - x} \left( {\frac{1}{2}} \right)_x \cos \left( {\pi x} \right),\\
\left( {x + \frac{1}{2}} \right)_k  &= \frac{{\left( {\frac{1}{2}} \right)_{k + x} }}{{\left( {\frac{1}{2}} \right)_x }}\label{uj5wt04}
\end{align}
and
\begin{equation}\label{zarppj5}
\left( {x - \frac{1}{2}} \right)_k  =\frac{2x-1}{2x+2k-1} \,\frac{{\left( {\frac{1}{2}} \right)_{k + x} }}{{\left( {\frac{1}{2}} \right)_x }}.
\end{equation}
\end{remark}
\begin{remark}
In view of~\eqref{uj5wt04} and~\eqref{zarppj5}, we have the following alternative but equivalent representations to~\eqref{bjv8z0f} and~\eqref{kebxljz}:
\begin{equation}
\left( {x + \frac{1}{2}} \right)_k  = \binom{{2x}}{x}^{ - 1} \,\frac{{k!}}{{2^{2k} }}\,\binom{{k + x}}{k}\,\binom{{2\left( {k + x} \right)}}{{k + x}}
\end{equation}
and
\begin{equation}
\left( {x - \frac{1}{2}} \right)_k  = \frac{{2x - 1}}{{2x + 2k - 1}}\,\binom{{2x}}{x}^{ - 1} \,\frac{{k!}}{{2^{2k} }}\,\binom{{k + x}}{k}\,\binom{{2\left( {k + x} \right)}}{{k + x}}.
\end{equation}

\end{remark}
\begin{lemma}
We have
\begin{equation}\label{ty98d9m}
\frac{d}{{dx}}\left( x \right)_y  = \left( x \right)_y \left( {H_{x + y - 1}  - H_{x - 1} } \right).
\end{equation}
\end{lemma}

\section{Generalizations of Ramanujan series for $1/\pi$}
\begin{lemma}\label{rutep3z}
If $k$ is a non-negative integer, then
\begin{equation}
\lim_{d\to\infty}{\frac{{\left( d \right)_k }}{{\left( { - d} \right)_k }}} = ( - 1)^k .
\end{equation}
\end{lemma}
\begin{proof}
Now
\begin{equation}
\left( { - d} \right)_k  = k!\binom{{ - d + k - 1}}{k} = ( - 1)^k k!\binom{{d}}{k}.
\end{equation}
But
\begin{equation}
\left( d \right)_k  = k!\binom{{d + k - 1}}{k} \sim k!\binom{{d}}{k}\text{ as $d\to\infty$}.
\end{equation}
Hence the result.
\end{proof}
\begin{lemma}\label{z2u9bwg}
If $a$ and $k$ are complex numbers, then
\begin{equation}
\frac{{\left( {1 + \frac{a}{2}} \right)_k }}{{\left( {\frac{a}{2}} \right)_k }} = \frac{{2k + a}}{a}.
\end{equation}
\end{lemma}
\begin{lemma}
For suitably bounded complex numbers $a$, $b$, and $c$, the following identity holds:
\begin{equation}\label{rbz4noe}
\sum_{k = 0}^\infty  {\frac{{( - 1)^k }}{{k!}}\frac{{\left( a \right)_k \left( b \right)_k \left( c \right)_k }}{{\left( {1 + a - b} \right)_k \left( {1 + a - c} \right)_k }}\left( {2k + a} \right)}  = \frac{{a\Gamma \left( {1 + a - b} \right)\Gamma \left( {1 + a - c} \right)}}{{\Gamma \left( {1 + a} \right)\Gamma \left( {1 + a - b - c} \right)}}.
\end{equation}
In particular,
\begin{equation}\label{q0xvexn}
\sum_{k = 0}^\infty  {\frac{{( - 1)^k }}{{k!^2 }}\frac{{\left( a \right)_k^2 \left( c \right)_k }}{{\left( {1 + a - c} \right)_k }}\left( {2k + a} \right)}  = \frac{{a\Gamma \left( {1 + a - c} \right)}}{{\Gamma \left( {1 + a} \right)\Gamma \left( {1 - c} \right)}}
\end{equation}
and
\begin{equation}\label{nw1xvnr}
\sum_{k = 0}^\infty  {( - 1)^k \frac{{\left( a \right)_k^3 }}{{k!^3 }}\left( {2k + a} \right)}  = \frac{{\sin \left( {\pi a} \right)}}{\pi }.
\end{equation}
\end{lemma}
\begin{proof}
Consider Dougall's well-poised hypergeometric series~\cite[p.~27]{bailey35}:
\begin{align}\label{ym6fvmk}
&\sum_{k = 0}^\infty  {\frac{1}{{k!}}\frac{{\left( a \right)_k \left( {1 + a/2} \right)_k \left( b \right)_k \left( c \right)_k \left( d \right)_k }}{{\left( {a/2} \right)_k \left( {1 + a - b} \right)_k \left( {1 + a - c} \right)_k \left( {1 + a - d} \right)_k }}}\nonumber\\
&\qquad  = \frac{{\Gamma \left( {1 + a - b} \right)\Gamma \left( {1 + a - c} \right)\Gamma \left( {1 + a - d} \right)\Gamma \left( {1 + a - b - c - d} \right)}}{{\Gamma \left( {1 + a} \right)\Gamma \left( {1 + a - b - c} \right)\Gamma \left( {1 + a - b - d} \right)\Gamma \left( {1 + a - c - d} \right)}},
\end{align}
which holds for suitably bounded complex numbers $a$, $b$, $c$, and $d$ or $\Re(1+a-b-c-d)>0$ if each of the four numbers has a finite magnitude.

Letting $d$ approach infinity in~\eqref{ym6fvmk} gives~\eqref{rbz4noe} by application of Lemmata~\ref{rutep3z} and~\ref{z2u9bwg}.
\end{proof}

\begin{remark}
Campbell~\cite{campbell26} also derived~\eqref{q0xvexn} using Mishev's transform. Identity~\eqref{nw1xvnr} is Dougall's identity~\cite[Equation (16)]{dougall06}.
\end{remark}

\begin{remark}
Identity~\eqref{rbz4noe} is also~\cite[Equation (3),p.~38]{bailey35}.
\end{remark}

\begin{theorem}
If $p$, $q$, and $r$ are complex numbers that are not half-integers and such that $q-p$ and $r-p$ are not non-positive integers, then
\begin{align}\label{czrzlst}
&\sum_{k = 0}^\infty  {( - 1)^k \frac{{\left( {\frac{1}{2}} \right)_k^3 }}{{k!^3 }}\,\frac{{\left( {4k + 1 - 2p} \right)}}{{\binom{{k - 1/2}}{p}\binom{{k - 1/2}}{q}\binom{{k - 1/2}}{r}\binom{{k + q - p}}{k}\binom{{k + r - p}}{k}}}}\nonumber\\
&\qquad= 2\,\frac{{\left( {\frac{1}{2} - p} \right)_{q + \frac{1}{2}} }}{{\left( {r - p} \right)_{q + \frac{1}{2}} }}\,\frac{{p!q!r!}}{{\left( {\frac{1}{2}} \right)_p \left( {\frac{1}{2}} \right)_q \left( {\frac{1}{2}} \right)_r }}\,\frac{{r - p}}{{\cos \left( {\pi p} \right)\cos \left( {\pi q} \right)\cos \left( {\pi r} \right)}} .
\end{align}
In particular,
\begin{equation}
\sum_{k = 0}^\infty  {( - 1)^k \frac{{\left( {\frac{1}{2}} \right)_k^3 }}{{k!^3 }}\frac{{\left( {4k + 1} \right)}}{{\binom{{k + r}}{k}\binom{{k - 1/2}}{r}}}}  = \frac{2}{\pi }\,\frac{{r!^2 }}{{\left( {\frac{1}{2}} \right)_r^2 }}\,\frac{1}{{\cos \left( \pi r \right)}}.
\end{equation}
\end{theorem}
\begin{proof}
Set $a=1/2-p$, $b=1/2-q$, and $c=1/2-r$ in~\eqref{rbz4noe} to obtain
\begin{align}\label{h7933kq}
&\sum_{k = 0}^\infty  {\frac{{( - 1)^k }}{{k!}}\,\frac{{\left( {1/2 - p} \right)_k \left( {1/2 - q} \right)_k \left( {1/2 - r} \right)_k }}{{\left( {1 - p + q} \right)_k \left( {1 - p + r} \right)_k }}\left( {4k + 1 - 2p} \right)}\nonumber\\
&\qquad  = \frac{{2\,\Gamma \left( {1 - p + q} \right)\Gamma \left( {1 - p + r} \right)}}{{\Gamma \left( {1/2 - p} \right)\Gamma \left( {1/2 - p + q + r} \right)}} .
\end{align}
Use~\eqref{l5fmdri} and note from~\eqref{ex_pochhammer} that
\begin{equation}\label{nly5agg}
\left( {1 + x} \right)_k  = k!\binom{{x + k}}{k},
\end{equation}
\begin{equation}
\frac{{\Gamma \left( {1 - p + q} \right)}}{{\Gamma \left( {\frac{1}{2} - p} \right)}} = \left( {\frac{1}{2} - p} \right)_{\frac{1}{2} + q} ,
\end{equation}
and
\begin{equation}
\frac{{\Gamma \left( {1 - p + r} \right)}}{{\Gamma \left( {\frac{1}{2} - p + q + r} \right)}} = \frac{{r - p}}{{\left( {r - p} \right)_{\frac{1}{2} + q} }}.
\end{equation}
\end{proof}
\begin{corollary}
If $\ell$, $m$, and $n$ are non-negative integers, then
\begin{align*}
&\sum_{k = 0}^\infty  {\frac{{( - 1)^k }}{{k!^3 }}\frac{{\left( {\frac{1}{2}} \right)_k^3 \left( {4k + 1 - 2\ell} \right)}}{{\prod_{j = 1}^\ell {\left( {2k - 2j + 1} \right)} \prod_{j = 1}^m {\left( {2k - 2j + 1} \right)} \prod_{j = 1}^n {\left( {2k - 2j + 1} \right)} \prod_{j = 1}^{m - \ell} {\left( {k + j} \right)} \prod_{j = 1}^{n - \ell} {\left( {k + j} \right)} }}}\\
&\qquad  = \frac{{( - 1)^{m + n} }}{{\left( {\frac{1}{2}} \right)_m \left( {\frac{1}{2}} \right)_n \left( {\frac{1}{2}} \right)_{m + n - \ell} 2^{\ell + m + n - 1} }}\,\frac{1}{\pi }.
\end{align*}
In particular,
\begin{align}
\sum_{k = 0}^\infty  {( - 1)^k \frac{{\left( {\frac{1}{2}} \right)_k^3 }}{{k!^3 }}\frac{{\left( {4k + 1} \right)}}{{\prod_{j = 1}^n {\left( {2k - 2j + 1} \right)\left( {k + j} \right)} }}}  &= \frac{{( - 1)^n }}{{\left( {\frac{1}{2}} \right)_n^2 2^{n - 1} }}\,\frac{1}{\pi },\\
\sum_{k = 0}^\infty  {( - 1)^k \frac{{\left( {\frac{1}{2}} \right)_k^3 }}{{k!^3 }}\frac{{\left( {4k + 1 - 2n} \right)}}{{\prod_{j = 1}^n {\left( {2k - 2j + 1} \right)^3 } }}}  &= \frac{1}{{\left( {\frac{1}{2}} \right)_n^3 2^{3n - 1} }}\,\frac{1}{\pi },
\end{align}
\begin{align}
&\sum_{k = 0}^\infty  {( - 1)^k \frac{{\left( {\frac{1}{2}} \right)_k^3 }}{{k!^3 }}\frac{{\left( {4k + 1 - 2m} \right)}}{{\prod_{j = 1}^m {\left( {2k - 2j + 1} \right)^2 } \prod_{j = 1}^n {\left( {2k - 2j + 1} \right)} \prod_{j = 1}^{n - m} {\left( {k + j} \right)} }}}\nonumber\\
&\qquad  = \frac{{( - 1)^{m + n} }}{{\left( {\frac{1}{2}} \right)_m \left( {\frac{1}{2}} \right)_n^2 2^{2m + n - 1} }}\,\frac{1}{\pi },
\end{align}
and
\begin{align}
&\sum_{k = 0}^\infty  {( - 1)^k \frac{{\left( {\frac{1}{2}} \right)_k^3 }}{{k!^3 }}\frac{{\left( {4k + 1 - 2\ell} \right)}}{{\prod_{j = 1}^\ell {\left( {2k - 2j + 1} \right)} \prod_{j = 1}^n {\left( {2k - 2j + 1} \right)^2 } \prod_{j = 1}^{n - \ell} {\left( {k + j} \right)^2 } }}}\nonumber\\
&\qquad  = \frac{1}{{\left( {\frac{1}{2}} \right)_{2n - \ell} \left( {\frac{1}{2}} \right)_n^2 2^{2n + \ell - 1} }}\,\frac{1}{\pi }.
\end{align}
\end{corollary}
\begin{proof}
Set $p=\ell$, $q=m$, and $r=n$ in~\eqref{h7933kq} and use~\eqref{mq7ny8m},~\eqref{wwymz81},~\eqref{zvg11tu}, and~\eqref{nly5agg}. Note that
\begin{equation}\label{cwqdxg9}
n!\binom{{k + n}}{n} = \prod_{j = 1}^n {\left( {k + j} \right)},\text{ etc} .
\end{equation}
\end{proof}

The identity stated in Lemma~\ref{edacukd} corresponds to setting $a=1/2=b$ in~\eqref{rbz4noe} and was derived by Hou et al.~\cite{hou24}. 

\begin{lemma}\label{edacukd}
If $c$ is a complex number, then
\begin{equation}\label{o68u3j7}
\sum_{k = 0}^\infty  {( - 1)^k \frac{{\left( {\frac{1}{2}} \right)_k^2 \left( c \right)_k }}{{k!^2 \left( {3/2 - c} \right)_k }}\left( {4k + 1} \right)}  = \frac{2}{{\sqrt \pi  }}\,\frac{{\Gamma \left( {3/2 - c} \right)}}{{\Gamma \left( {1 - c} \right)}}= \frac{2}{{\sqrt \pi  }}(1-c)_{1/2}.
\end{equation}
\end{lemma}

\begin{remark}
The special case of~\eqref{o68u3j7} with $c$ a non-negative integer was proved by Ekhad and Zeilberger~\cite{ekhad94} using the $WZ$ technique.
\end{remark}

\begin{theorem}
If $r$ is a complex number that is not a half-integer, then
\begin{equation}\label{r1kycoj}
\sum_{k = 0}^\infty  {( - 1)^k \frac{{\left( {\frac{1}{2}} \right)_k^3 }}{{k!^3 }}\frac{{\left( {4k + 1} \right)\left( {2O_{k - r}  + H_{k + r} } \right)}}{{\binom{{k + r}}{k}\binom{{k - 1/2}}{r}}}}  = \frac{2}{\pi }\,\frac{{r!^2 }}{{\left( {\frac{1}{2}} \right)_r^2 }}\,\frac{{\left( {4O_r  - 2\ln 2} \right)}}{{\cos \left( \pi r \right)}}.
\end{equation}
\end{theorem}
\begin{proof}
Differentiate~\eqref{o68u3j7} with respect to $c$ using~\eqref{ty98d9m} and replace $c$ with $1/2-r$ to get
\begin{equation}\label{hr8o23s}
\sum_{k = 0}^\infty  {\frac{{( - 1)^k \left( {\frac{1}{2}} \right)_k^2 \left( {\frac{1}{2} - r} \right)_k \left( {4k + 1} \right)\left( {H_{k - r - 1/2}  + H_{k + r} } \right)}}{{k!^2 \left( {r + 1} \right)_k }}}  = \frac{2}{{\sqrt \pi  }}\left( {r + \frac{1}{2}} \right)_{1/2} \left( {H_{ - r - 1/2}  + H_{r - 1/2} } \right),
\end{equation}
from which~\eqref{r1kycoj} follows after using~\eqref{nerd5x1},~\eqref{l5fmdri},~\eqref{nly5agg}, and the fact that
\begin{equation}\label{g20r5hg}
\left( {r + \frac{1}{2}} \right)_{\frac{1}{2}}  = \frac{{\Gamma \left( {r + 1} \right)}}{{\Gamma \left( {r + \frac{1}{2}} \right)}} = \frac{{r!}}{{\left( {\frac{1}{2}} \right)_r \sqrt \pi  }}.
\end{equation}
\end{proof}
\begin{corollary}
If $n$ is a non-negative integer, then
\begin{equation*}
\sum_{k = 0}^\infty  {\frac{{( - 1)^k \left( {\frac{1}{2}} \right)_k^3 \left( {4k + 1} \right)\left( {2O_{k - n}  + H_{k + n} } \right)}}{{k!^3 \prod_{j = 1}^n {\left( {k + j} \right)\left( {2k - 2j + 1} \right)} }}}  = \frac{{( - 1)^n }}{{2^{n - 2} }}\frac{1}{{\left( {\frac{1}{2}} \right)_n^2 }}\left( {2O_n  - \ln 2} \right)\,\frac{1}{\pi }.
\end{equation*}
In particular,
\begin{equation*}
\sum_{k = 0}^\infty  {( - 1)^k \frac{{\left( {\frac{1}{2}} \right)_k^3 }}{{k!^3 }}\left( {4k + 1} \right)H_{2k} }  =  - \frac{2}{\pi }\ln 2.
\end{equation*}
\end{corollary}
\begin{proof}
Set $r=n$ a non-negative integer in~\eqref{hr8o23s}.
\end{proof}

Lemma~\ref{nnoqyfq} was derived by Hou et al.~\cite{hou24} using a hypergeometric transformation formula due to Chu and Zhang~\cite[Theorem 9]{chu14}; and facilitates the derivation of~\eqref{gen2}.

\begin{lemma}\label{nnoqyfq}
If $c$ is a suitably bounded complex number, then
\begin{equation}\label{k1x81g9}
\sum_{k = 0}^\infty  {\frac{1}{{4^k }}\frac{{\left( {\frac{1}{2}} \right)_k \left( c \right)_k \left( {1 - c} \right)_k }}{{k!^2 \left( {3/2 - c} \right)_k }}\,\frac{{3k - c + 1}}{2}}  = \frac{1}{{\sqrt \pi  }}\,(1-c)_{1/2}.
\end{equation}
\end{lemma}
\begin{theorem}
If $r$ is a complex number that is not a half-integer, then
\begin{equation}
\sum_{k = 0}^\infty  {\frac{{\left( {\frac{1}{2}} \right)_k^3 }}{{4^k k!^3 }}\frac{{\binom{{2k + 2r}}{{2r}}}}{{\binom{{k + r}}{r}^2 \binom{{k - 1/2}}{r}}}\left( {6k + 1 + 2r} \right)}  = \frac{4}{\pi }\frac{{r!^2 }}{{\left( {\frac{1}{2}} \right)_r^2 }}\,\frac{{1 }}{{\cos\left( {\pi r} \right)}}.
\end{equation}
\end{theorem}
\begin{proof}
Set $c=1/2-r$ in~\eqref{k1x81g9} to obtain
\begin{equation}\label{w4w5kv4}
\sum_{k = 0}^\infty  {\frac{1}{{4^k }}\frac{{\left( {\frac{1}{2}} \right)_k \left( {\frac{1}{2} - r} \right)_k \left( {\frac{1}{2} + r} \right)_k }}{{k!^2 \left( {r + 1} \right)_k }}\left( {6k + 1 + 2r} \right)}  = \frac{4}{{\sqrt \pi  }}\left( {r + \frac{1}{2}} \right)_{1/2} .
\end{equation}
Now use~\eqref{l5fmdri},~\eqref{bjv8z0f},~\eqref{nly5agg}, and~\eqref{g20r5hg}.
\end{proof}
\begin{corollary}
If $n$ is a non-negative integer, then
\begin{equation*}
\sum_{k = 0}^\infty  {\frac{{\left( {\frac{1}{2}} \right)_k^3 }}{{4^k k!^3 }}\frac{{\prod_{j = 1}^{2n} {\left( {2k + j} \right)} }}{{\prod_{j = 1}^n {\left( {k + j} \right)^2 \left( {2k - 2j + 1} \right)} }}\left( {6k + 1 + 2n} \right)}  = \frac{{( - 1)^n\,2^{n+2} }}{{\left( {\frac{1}{2}} \right)_n }}\, \frac1{\pi }.
\end{equation*}
In particular,
\begin{equation}
\sum_{k = 0}^\infty  {\frac{{\left( {\frac{1}{2}} \right)_k^3 }}{{4^k k!^3 }}\frac{{\left( {2k + 1} \right)^2 }}{{\left( {k + 1} \right)\left( {2k - 1} \right)}}}  =  - \frac{8}{{3\pi }}.
\end{equation}
\end{corollary}
\begin{proof}
Set $r=n$ in~\eqref{w4w5kv4} and use~\eqref{mq7ny8m},~\eqref{t0khqm8},~\eqref{nly5agg}, and~\eqref{g20r5hg}.
\end{proof}
\begin{theorem}
If $r$ is a complex number that is not a half-integer, then
\begin{align}\label{yoctbpp}
&\sum_{k = 0}^\infty  {\frac{1}{{4^k }}\frac{{\left( {\frac{1}{2}} \right)_k^3 }}{{k!^3 }}\frac{{\binom{{2k + 2r}}{{2r}}}}{{\binom{{k + r}}{r}^2 \binom{{k - 1/2}}{r}}}\left( {\left( {6k + 2r + 1} \right)\left( {H_{k - r - 1/2}  - H_{k + r - 1/2}  + H_{k + r} } \right) - 2} \right)}\nonumber\\ 
&\qquad = \frac{{r!^2 }}{{\left( {\frac{1}{2}} \right)_r^2 }}\,\frac{{H_{ - r - 1/2} }}{{\cos \left( {\pi r} \right)}}\,\frac{4}{\pi }.
\end{align}
\end{theorem}
\begin{proof}
Differentiate~\eqref{k1x81g9} with respect to $c$ to obtain, after some rearrangement,
\begin{align}
&\sum_{k = 0}^\infty  {\frac{{\left( {\frac{1}{2}} \right)_k \left( c \right)_k \left( {1 - c} \right)_k }}{{4^kk!^2 \left( {\frac32 - c} \right)_k }}\left( {\frac{{3k - c + 1}}{2}\left( {H_{c + k - 1}  - H_{ - c + k}  + H_{1/2 - c + k} } \right) - \frac{1}{2}} \right)}\nonumber\\
&\qquad  = \frac{{H_{c - 1} }}{{\sqrt \pi  }}\left( {1 - c} \right)_{1/2} .
\end{align}
Set $c=1/2-r$ to get
\begin{align}\label{pehnhlp}
&\sum_{k = 0}^\infty  {\frac{{\left( {\frac{1}{2}} \right)_k \left( { - r + \frac{1}{2}} \right)_k \left( {r + \frac{1}{2}} \right)_k }}{{4^kk!^2 \left( {r + 1} \right)_k }}\left( {\frac{1}{2}\left( {3k + r + \frac{1}{2}} \right)\left( {H_{k - r - 1/2}  - H_{k + r - 1/2}  + H_{k + r} } \right) - \frac{1}{2}} \right)}\nonumber\\
&\qquad  = \frac{{H_{ - r - 1/2} }}{{\sqrt \pi  }}\left( {r + \frac{1}{2}} \right)_{1/2} ;
\end{align}
and hence~\eqref{yoctbpp} after using~\eqref{l5fmdri} and~\eqref{bjv8z0f}.
\end{proof}
\begin{corollary}
If $n$ is a non-negative integer, then
\begin{align*}
&\sum_{k = 0}^\infty  {\frac{{\left( {\frac{1}{2}} \right)_k^3 }}{{4^k k!^3 }}\frac{{\prod_{j = 1}^{2n} {\left( {2k + j} \right)} }}{{\prod_{j = 1}^n {\left( {k + j} \right)^2 \left( {2k - 2j + 1} \right)} }}\left( {\left( {6k + 1 + 2n} \right)\left( {2O_{k - n}  - 2O_{k + n}  + H_{k + n} } \right) - 2} \right)}\\
&\qquad  =( - 1)^n\,\frac{{2^{n + 3} }}{{\left( {\frac{1}{2}} \right)_n }}\left( {O_n  - \ln 2} \right)\frac{1}{\pi }.
\end{align*}
In particular,
\begin{equation}
\sum_{k = 0}^\infty  {\frac{{\left( {\frac{1}{2}} \right)_k^3 }}{{4^k k!^3 }}\left( {\left( {6k + 1} \right)H_k  - 2} \right)}  =  - \frac{{8\ln 2}}{\pi }.
\end{equation}
\end{corollary}
\begin{proof}
Set $r=n$ in~\eqref{pehnhlp}, use
\begin{equation*}
H_{k - n - 1/2}  - H_{k + n - 1/2}  = 2\left( {O_{k - n}  - O_{k + n} } \right)
\end{equation*}
and identities~\eqref{mq7ny8m} and~\eqref{t0khqm8}.
\end{proof}
\section{Additional results: discovering Ramanujan-like series for $1/\pi$}

By expressing the parameters of a (very well-poised) hypergeometric series as half-integers and applying Lemma~\ref{w00yljy}, the Ramanujan-like series associated with such hypergeometric series may be discovered. We give further illustrations of the method by deriving families of Ramanujan-like series associated with an identity of Lavoie~\cite[Equation (22)]{lavoie66} stated in Lemma~\ref{qobi756}, two identities found in Bailey's book~\cite{bailey35} and an identity of Campbell~\cite{campbell26}, given here in lemma~\ref{q52sf7y}.
\begin{lemma}\label{qobi756}
If $x$ and $y$ are suitably bounded complex numbers, then
\begin{equation}\label{y7sext5}
\sum_{k = 0}^\infty  {\frac{{\left( { - \frac{1}{2}} \right)_k \left( x \right)_k }}{{k!\left( {x + \frac 32} \right)_k }}\left( {1 + 2y + 2k} \right)}  = \frac{{\left( {1 - x + 2y} \right)\left( {2x + 1} \right)}}{{\left(x + 1\right)}!}\,\left(\frac12\right)_x.
\end{equation}
\end{lemma}
\begin{theorem}
If $r$ and $y$ are complex numbers such that $r\ne\pm 1/2$, then
\begin{equation}\label{g97hg7l}
\sum_{k = 0}^\infty  {\frac{{\left( {\frac{1}{2}} \right)_k^2 }}{{k!^2 }}\frac{{\binom{{2k + 2r}}{{2r}}}}{{\binom{{k + r}}{r}^2 }}\frac{{\left( {y + k} \right)}}{{\left( {2k + 2r - 1} \right)\left( {2k - 1} \right)}}}  = \frac{{r!\left( {1 - 2r + 4y} \right)}}{{\left( {\frac{1}{2}} \right)_r \left( {1 + 2r} \right)\left( {1 - 2r} \right)}}\,\frac{1}{\pi }.
\end{equation}
In particular,
\begin{equation}
\sum_{k = 0}^\infty  {\left( {\frac{{\left( {\frac{1}{2}} \right)_k }}{{k!\left( {2k - 1} \right)}}} \right)^2 \left( {y + k} \right)}  = \frac{{1+4y}}{\pi },
\end{equation}
with the special value
\begin{equation*}
\sum_{k = 0}^\infty  {\left( {\frac{{\left( {\frac{1}{2}} \right)_k }}{{k!\left( {2k - 1} \right)}}} \right)^2 \left( {2k + 1} \right)}  = \frac{6}{\pi }.
\end{equation*}
\end{theorem}
\begin{proof}
Set $x=r-1/2$ in~\eqref{y7sext5} to get
\begin{equation*}
\sum_{k = 0}^\infty  {\frac{{\left( {r - \frac{1}{2}} \right)_k \left( { - \frac{1}{2}} \right)_k }}{{k!\left( {r + 1} \right)_k }}\left( {y + k} \right)}  = \frac{{1 - 2r + 4y}}{{\left( {1 + 2r} \right)}}\,\frac{{r!}}{{\left( {\frac{1}{2}} \right)_r }}\,\frac{1}{\pi },
\end{equation*}
and use~\eqref{kebxljz} and the fact that
\begin{equation}\label{ze82g3l}
\left( { - \frac{1}{2}} \right)_k  = \frac{1}{{1 - 2k}}\left( {\frac{1}{2}} \right)_k 
\end{equation}
to arrive at~\eqref{g97hg7l}.
\end{proof}
\begin{corollary}
If $r$ is a complex number that is not a negative integer and $r\ne-1/2$, then
\begin{equation}\label{j8y9syu}
\sum_{k = 1}^\infty  {\frac{{\left( {\frac{1}{2}} \right)_k^2 }}{{k!\left( {k - 1} \right)!}}\frac{{\binom{{2k + 2r}}{{2r}}}}{{\binom{{k + r}}{r}^2 }}\frac{1}{{\left( {2k + 2r - 1} \right)\left( {2k - 1} \right)}}}  = \frac{{r!}}{{\left( {\frac{1}{2}} \right)_r \left( {1+2r} \right)}}\,\frac{1}{\pi }
\end{equation}
and
\begin{equation}\label{i4ivc6y}
\sum_{k = 0}^\infty  {\frac{{\left( {\frac{1}{2}} \right)_k^2 }}{{k!^2 }}\frac{{\binom{{2k + 2r}}{{2r}}}}{{\binom{{k + r}}{r}^2 }}\frac{1}{{\left( {2k + 2r - 1} \right)\left( {2k - 1} \right)}}}  = \frac{{r!}}{{\left( {\frac{1}{2}} \right)_r \left( {1 + 2r} \right)\left( {1 - 2r} \right)}}\,\frac{4}{\pi },\quad r\ne1/2.
\end{equation}
In particular,
\begin{equation}
\sum_{k = 1}^\infty  {\left( {\frac{{\left( {\frac{1}{2}} \right)_k }}{{2k - 1}}} \right)^2 \frac{1}{{k!\left( {k - 1} \right)!}}}  = \frac{1}{\pi }
\end{equation}
and
\begin{equation}
\sum_{k = 0}^\infty  {\left( {\frac{{\left( {\frac{1}{2}} \right)_k }}{{k!\left( {2k - 1} \right)}}} \right)^2 }  = \frac{4}{\pi }.
\end{equation}
\end{corollary}
\begin{proof}
Plug $y=0$ in~\eqref{g97hg7l} to obtain~\eqref{j8y9syu} while~\eqref{i4ivc6y} now follows from~\eqref{j8y9syu} and~\eqref{g97hg7l}.
\end{proof}
\begin{remark}
Identity~\eqref{i4ivc6y} can formally be obtained by differentiating both sides of~\eqref{g97hg7l} with respect to $y$.
\end{remark}
\begin{theorem}
If $n$ is a non-negative integer, then
\begin{align}\label{eiv1sal}
&\sum_{k = 0}^\infty  {\frac{{\left( {\frac{1}{2}} \right)_k^2 }}{{k!^2 }}\frac{{\binom{{2k + 2n}}{{2n}}}}{{\binom{{k + n}}{n}^2 }}\frac{{\left( {2k + 1} \right)}}{{\left( {2k + 2n - 1} \right)\left( {2k - 1} \right)}}\left( {H_{k + n}  - 2O_{k + n - 1} } \right)}\nonumber\\
&\qquad  = \frac{{n!4\left( {2n - 3} \right)\ln 2}}{{\left( {1 - 2n} \right)\left( {1 + 2n} \right)\left( {\frac{1}{2}} \right)_n }}\,\frac{1}{\pi } + \frac{{n!4\left( {4n^2  - 12n + 1} \right)}}{{\left( {1 - 2n} \right)^2 \left( {1 + 2n} \right)^2 \left( {\frac{1}{2}} \right)_n }}\,\frac{1}{\pi }.
\end{align}
In particular,
\begin{equation*}
\sum_{k = 0}^\infty  {\frac{{\left( {\frac{1}{2}} \right)_k^2 }}{{k!^2 }}\frac{{\left( {2k + 1} \right)}}{{\left( {2k - 1} \right)^2 }}\left( {H_k  - 2O_{k - 1} } \right)}  =  - \frac{{12\ln 2}}{\pi } + \frac{4}{\pi }.
\end{equation*}
\end{theorem}
\begin{proof}
Set $y=0$ in~\eqref{y7sext5} to obtain
\begin{equation*}
\sum_{k = 0}^\infty  {\frac{{\left( x \right)_k \left( { - \frac{1}{2}} \right)_k }}{{k!\left( {\frac{3}{2} + x} \right)_k }}\left( {2k + 1} \right)}  = \frac{{\left( {1 - x} \right)}}{{\left( {1 + x}\, \right)}}\frac{{\left( {2x + 1} \right)}}{{\left( {\frac{1}{2} + x} \right)_{\frac{1}{2}} \sqrt \pi  }},
\end{equation*}
which, by differentiating with respect to $x$ gives, after some algebra,
\begin{equation*}
\sum_{k = 0}^\infty  {\frac{{\left( x \right)_k \left( { - \frac{1}{2}} \right)_k }}{{k!\left( {\frac{3}{2} + x} \right)_k }}\left( {2k + 1} \right)\left( {H_{1/2 + x + k}  - H_{x + k - 1} } \right)}  = \frac{{\left( {2x + 1} \right)}}{{x\left( {1 + x} \right)^2 }}\,\frac{{\left( {1 + 2x - x^2 } \right)}}{{\left( {\frac{1}{2} + x} \right)_{\frac{1}{2}} \sqrt \pi  }}.
\end{equation*}
Now set $x=n-1/2$, use~\eqref{nerd5x1} and rearrange to get
\begin{align*}
&\sum_{k = 0}^\infty  {\frac{{\left( {n - \frac{1}{2}} \right)_k \left( { - \frac{1}{2}} \right)_k }}{{k!\left( {1 + n} \right)_k }}\left( {2k + 1} \right)\left( {H_{k + n}  - 2O_{k + n - 1} } \right)}\\
&\qquad  = \frac{{4n\left( {2n - 3} \right)\ln 2}}{{\left( {1 + 2n} \right)n_{\frac{1}{2}} \sqrt \pi  }} + \frac{{4n\left( {4n^2  - 12n + 1} \right)}}{{\left( {1 - 2n} \right)\left( {1 + 2n} \right)^2 n_{\frac{1}{2}} \sqrt \pi  }},
\end{align*}
from which~\eqref{eiv1sal} follows after using~\eqref{kebxljz},~\eqref{ze82g3l}, and the fact that
\begin{equation*}
n_{\frac{1}{2}}  = \frac{{\Gamma \left( {n + 1/2} \right)}}{{\Gamma \left( n \right)}} = \frac{{\left( {\frac{1}{2}} \right)_n \sqrt \pi  }}{{\left( {n - 1} \right)!}}.
\end{equation*}
\end{proof}

Our next families of Ramanujan-like series for $1/\pi$ presented in Theorems~\ref{oae5ink} and~\ref{r5cu3is} and Corollaries~\ref{iu3zyaw} and~\ref{tldwdgd} come from Equation (1.1) of Bailey~\cite[p.~30]{bailey35} stated here in Lemma~\ref{crjrti3}.
\begin{lemma}\label{crjrti3}
If $a$, $m$, and $w$ are suitably bounded complex numbers, then
\begin{equation}\label{xw71jsn}
\sum_{k = 0}^\infty  {\frac{{\left( a \right)_k \left( { - m} \right)_k }}{{k!\left( w \right)_k }}\left( {2k + a} \right)}  = \frac{{a\left( {w - a - 1 - m} \right)\left( {w - a} \right)_{m - 1} }}{{\left( w \right)_m }}.
\end{equation}
\end{lemma}
\begin{theorem}\label{oae5ink}
If $p$, $q$, and $r$ are complex numbers that are not negative integers and not half-integers, then
\begin{align}
&\sum_{k = 0}^\infty  {\frac{{\left( {\frac{1}{2}} \right)_k^2 }}{{k!^2 }}\,\frac{{\left( {1 - 2p + 4k} \right)}}{{\binom{{k - 1/2}}{p}\binom{{k - 1/2}}{q}\binom{{k + r}}{k}}}}\nonumber\\
&\qquad  = \frac{{p!q!r!}}{{\cos \left( {\pi p} \right)\cos \left( {\pi q} \right)}}\,\frac{{\left( {p + q + r - 2} \right)!\left( {p - q + r} \right)\left( {1 - 2p} \right)}}{{\left( {\frac{1}{2}} \right)_p \left( {\frac{1}{2}} \right)_q \left( {\frac{1}{2}} \right)_{p + r} \left( {\frac{1}{2}} \right)_{q + r} }}\,\frac{1}{\pi }.
\end{align}
In particular,
\begin{align}
\sum_{k = 0}^\infty  {\frac{{\left( {\frac{1}{2}} \right)_k^2 }}{{k!^2 }}\,\frac{{\left( {1 + 4k} \right)}}{{\binom{{k - 1/2}}{q}\binom{{k + r}}{k}}}}  &= \frac{{q!r!}}{{\cos \left( {\pi q} \right)}}\frac{{\left( {r + q - 2} \right)!\left( {r - q} \right)}}{{\left( {\frac{1}{2}} \right)_q \left( {\frac{1}{2}} \right)_r \left( {\frac{1}{2}} \right)_{q + r} }}\,\frac{1}{\pi },\\
\sum_{k = 0}^\infty  {\frac{{\left( {\frac{1}{2}} \right)_k^2 }}{{k!^2 }}\,\frac{{\left( {1 + 4k} \right)}}{{\binom{{k + r}}{k}}}}  &= \frac{{\left( {r - 2} \right)!r!r}}{{\left( {\frac{1}{2}} \right)_r^2 }}\,\frac{1}{\pi },\\
\sum_{k = 0}^\infty  {\frac{{\left( {\frac{1}{2}} \right)_k^2 }}{{k!^2 }}\,\frac{{\left( {1 + 4k} \right)}}{{\binom{{k - 1/2}}{r}}}}  &= -\frac{{\left( {r - 2} \right)!r!r}}{{\left( {\frac{1}{2}} \right)_r^2 \cos \left( {\pi r} \right)}}\,\frac{1}{\pi },
\end{align}
and
\begin{equation}
\sum_{k = 0}^\infty  {\frac{{\left( {\frac{1}{2}} \right)_k^2 }}{{k!^2 }}\,\frac{{\left( {1 - 2r + 4k} \right)}}{{\binom{{k - 1/2}}{r}}}}  = \frac{{r!\left( {r - 2} \right)!r\left( {1 - 2r} \right)}}{{\left( {\frac{1}{2}} \right)_r^2 \cos \left( {\pi r} \right)}}\,\frac{1}{\pi }.
\end{equation}
\end{theorem}
\begin{proof}
Set $a=1/2-p$, $m=q-1/2$, and $w=1+r$ in~\eqref{xw71jsn} to get
\begin{equation}\label{wfcpufm}
\sum_{k = 0}^\infty  {\frac{{\left( {1/2 - p} \right)_k \left( {1/2 - q} \right)_k }}{{k!\left( {1 + r} \right)_k }}\left( {1 - 2p + 4k} \right)}  = \frac{{\left( {1/2 + p + r} \right)_{q - 3/2} }}{{\left( {1 + r} \right)_{q - 1/2} }}\left( {p - q + r} \right)\left( {1 - 2p} \right).
\end{equation}
Now use~\eqref{l5fmdri},~\eqref{nly5agg}, and the fact that
\begin{align}\label{vhjgpsi}
\frac{{\left( {1/2 + p + r} \right)_{q - 3/2} }}{{\left( {1 + r} \right)_{q - 1/2} }} &= \frac{{\Gamma \left( {p + q + r - 1} \right)\Gamma \left( {r + 1} \right)}}{{\Gamma \left( {1/2 + p + r} \right)\Gamma \left( {1/2 + q + r} \right)}}\nonumber\\
& = \frac{{\left( {p + q + r - 2} \right)!r!}}{{\left( {1/2} \right)_{p + r} \left( {1/2} \right)_{q + r} }}\,\frac{1}{\pi },
\end{align}
by virtue of~\eqref{zvg11tu} and the definition of the Gamma function.
\end{proof}
\begin{corollary}\label{iu3zyaw}
If $\ell$, $m$, and $n$ are non-negative integers, then
\begin{align*}
&\sum_{k = 0}^\infty  {\frac{{\left( {\frac{1}{2}} \right)_k^2 }}{{k!^2 }}\,\frac{{\left( {1 - 2\ell + 4k} \right)}}{{\prod_{j = 1}^\ell {\left( {2k - 2j + 1} \right)} \prod_{j = 1}^m \left({2k - 2j + 1}\right) \prod_{j = 1}^n {\left( {k + j} \right)} }}}\\
&\qquad  = \frac{{( - 1)^{\ell + m} }}{{2^{\ell + m} }}\,\frac{{\left( {\ell + m + n - 2} \right)!\left( {\ell - m + n} \right)\left(1-2\ell\right)}}{{\left( {\frac{1}{2}} \right)_\ell \,\left( {\frac{1}{2}} \right)_m \left( {\frac{1}{2}} \right)_{\ell + n} \left( {\frac{1}{2}} \right)_{m + n} }}\,\frac{1}{\pi }.
\end{align*}
In particular,
\begin{align}
\sum_{k = 0}^\infty  {\frac{{\left( {\frac{1}{2}} \right)_k^2 }}{{k!^2 }}\,\frac{{\left( {1 + 4k} \right)}}{{\prod_{j = 1}^m {\left(2k - 2j + 1\right)} \prod_{j = 1}^n {\left( {k + j} \right)} }}}  &= \frac{{( - 1)^m }}{{2^m }}\,\frac{{\left( {m + n - 2} \right)!\left( {n - m} \right)}}{{\left( {\frac{1}{2}} \right)_m \left( {\frac{1}{2}} \right)_n \left( {\frac{1}{2}} \right)_{m + n} }}\,\frac{1}{\pi },\\
\sum_{k = 0}^\infty  {\frac{{\left( {\frac{1}{2}} \right)_k^2 }}{{k!^2 }}\,\frac{{\left( {1 + 4k} \right)}}{{\prod_{j = 1}^n {\left( {k + j} \right)} }}}  &= \frac{{\left( {n - 2} \right)!n}}{{\left( {\frac{1}{2}} \right)_n^2 }}\,\frac{1}{\pi },\\
\sum_{k = 0}^\infty  {\frac{{\left( {\frac{1}{2}} \right)_k^2 }}{{k!^2 }}\,\frac{{\left( {1 + 4k} \right)}}{{\prod_{j = 1}^n {\left( {2k - 2j + 1} \right)} }}}  &= \frac{{( - 1)^{n + 1} }}{{2^n }}\,\frac{{\left( {n - 2} \right)!n}}{{\left( {\frac{1}{2}} \right)_n^2 }}\,\frac{1}{\pi },
\end{align}
and
\begin{equation}
\sum_{k = 0}^\infty  {\frac{{\left( {\frac{1}{2}} \right)_k^2 }}{{k!^2 }}\,\frac{{\left( {1 - 2n + 4k} \right)}}{{\prod_{j = 1}^n {\left( {2k - 2j + 1} \right)} }}}  = \frac{{( - 1)^n }}{{2^n }}\,\frac{{\left( {n - 2} \right)!n\left( {1 - 2n} \right)}}{{\left( {\frac{1}{2}} \right)_n^2 }}\,\frac{1}{\pi }.
\end{equation}
\end{corollary}
\begin{proof}
Set $p=\ell$, $q=m$, and $r=n$ in~\eqref{wfcpufm} and use~\eqref{mq7ny8m},~\eqref{cwqdxg9}, and~\eqref{vhjgpsi}.
\end{proof}
Explicit examples from Corollary~\ref{iu3zyaw} include
\begin{align}
\sum_{k = 0}^\infty  {\frac{{\left( {\frac{1}{2}} \right)_k^2 }}{{k!^2 }}\,\frac{{\left( {1 + 4k} \right)}}{{\left( {k + 1} \right)\left( {k + 2} \right)}}}  &= \frac{{32}}{{9\pi }},\\
\sum_{k = 0}^\infty  {\frac{{\left( {\frac{1}{2}} \right)_k^2 }}{{k!^2 }}\,\frac{{\left( {1 + 4k} \right)}}{{\left( {2k - 1} \right)\left( {2k - 3} \right)}}}  &=  - \frac{8}{{9\pi }},\\
\sum_{k = 0}^\infty  {\frac{{\left( {\frac{1}{2}} \right)_k^2 }}{{k!^2 }}\,\frac{{\left( {4k - 3} \right)}}{{\left( {2k - 1} \right)\left( {2k - 3} \right)}}}  &=  - \frac{8}{{3\pi }},
\end{align}
and
\begin{equation}
\sum_{k = 0}^\infty  {\frac{{\left( {\frac{1}{2}} \right)_k^2 }}{{k!^2 }}\,\frac{1}{{\left( {2k - 1} \right)\left( {2k - 3} \right)}}}  = \frac{4}{{9\pi }}.
\end{equation}
\begin{theorem}\label{r5cu3is}
If $p$, $q$, and $r$ are complex numbers that are not negative integers and not half-integers, then
\begin{align}
&\sum_{k = 0}^\infty  {\frac{{\left( {\frac{1}{2}} \right)_k^2 }}{{k!^2 }}\frac{{\left( {1 - 2p + 4k} \right)H_{k + r} }}{{\binom{{k - 1/2}}{p}\binom{{k - 1/2}}{q}\binom{{k + r}}{k}}}}\nonumber\\
&\qquad  = \frac{{p!q!r!\left( {p + q + r - 2} \right)!}}{{\cos \left( {\pi p} \right)\cos \left( {\pi q} \right)\left( {\frac{1}{2}} \right)_p \left( {\frac{1}{2}} \right)_q \left( {\frac{1}{2}} \right)_{r + p} \left( {\frac{1}{2}} \right)_{r + q} }}\nonumber\\
&\qquad\qquad\qquad\times\left( {1 + \left( {p - q + r} \right)\left( {H_{p + q + r - 2}  - H_{p + r - 1/2}  - H_{q + r - 1/2} } \right)} \right)\,\frac{1}{\pi }.
\end{align}
\end{theorem}
\begin{proof}
Differentiate~\eqref{wfcpufm} with respect to $r$, to get, after some algebra,
\begin{align}\label{tabjqbw}
&\sum_{k = 0}^\infty  {\frac{{\left( {1/2 - p} \right)_k \left( {1/2 - q} \right)_k }}{{k!\left( {1 + r} \right)_k }}\left( {1 - 2p + 4k} \right)H_{r + k} }\nonumber\\
&\qquad  = \frac{{\left( {1/2 + p + r} \right)_{q - 3/2} }}{{\left( {1 + r} \right)_{q - 1/2} }}\left( {2p - 1} \right)\nonumber\\
&\qquad\qquad\times\left( {1 + \left( {p - q + r} \right)\left( {H_{p + q + r - 2}  - H_{p + r - 1/2}  - H_{q + r - 1/2} } \right)} \right).
\end{align}
Now use~\eqref{l5fmdri},~\eqref{nly5agg}, and~\eqref{vhjgpsi}.
\end{proof}
\begin{corollary}\label{tldwdgd}
If $\ell$, $m$, and $n$ are non-negative integers, then
\begin{align}
&\sum_{k = 0}^\infty  {\frac{{\left( {\frac{1}{2}} \right)_k^2 }}{{k!^2 }}\,\frac{{\left( {1 - 2\ell + 4k} \right)H_{k + n} }}{{\prod_{j = 1}^\ell {\left( {2k - 2j + 1} \right)} \prod_{j = 1}^m {\left(2k - 2j + 1\right)} \prod_{j = 1}^n {\left( {k + j} \right)} }}}\nonumber\\
&\qquad  = \frac{{( - 1)^{\ell + m} }}{{2^{\ell + m} }}\,\frac{{\left( {\ell + m + n - 2} \right)!\left( {1 - 2\ell} \right)}}{{\left( {\frac{1}{2}} \right)_\ell \left( {\frac{1}{2}} \right)_m \left( {\frac{1}{2}} \right)_{\ell + n} \left( {\frac{1}{2}} \right)_{m + n} }}\nonumber\\
&\qquad\qquad\times\left( {1 + \left( {\ell - m + n} \right)\left( {H_{\ell + m + n - 2}  - 2O_{\ell + n}  - 2O_{m + n}  + 4\ln 2} \right)} \right)\,\frac{1}{\pi }.
\end{align}
In particular,
\begin{equation}
\sum_{k = 0}^\infty  {\frac{{\left( {\frac{1}{2}} \right)_k^2 }}{{k!^2 }}\,\frac{{\left( {1 + 4k} \right)H_k }}{{\left( {2k - 1} \right)\left( {2k - 3} \right)}}}  =  - \frac{{76}}{{27\pi }} + \frac{{32}}{{9\pi }}\ln 2.
\end{equation}
\end{corollary}
\begin{proof}
Set $p=\ell$, $q=m$, and $r=n$ in~\eqref{tabjqbw} and use~\eqref{nerd5x1},~\eqref{mq7ny8m},~\eqref{cwqdxg9}, and~\eqref{vhjgpsi}.
\end{proof}
The next hypergeometric series from which we will derive families of Ramanujan-like series for $1/\pi$ comes from Bailey~\cite[p.~98, Examples 9]{bailey35} and is stated in Lemma~\ref{jj3n02u}.
\begin{lemma}\label{jj3n02u}
If $x$ and $y$ are suitably bounded complex numbers, then
\begin{align}\label{iuxoyr8}
&\sum_{k = 0}^\infty  {\frac{1}{{k!}}\,\frac{{\left( {x + y} \right)_k \left( {x - y} \right)_k }}{{\left( {2x} \right)_k }}\,\frac{1}{{2x - 1 + 2k}}}\nonumber\\
&\qquad  = \frac{1}{2}\,\frac{1}{{2x - 1}}\,\frac{{\Gamma \left( {2x} \right)}}{{\Gamma \left( {x + y} \right)\Gamma \left( {x - y} \right)}}\,\left( {H_{\frac{{x + y - 1}}{2}}  + H_{\frac{{x - y - 1}}{2}}  - H_{\frac{{x + y - 2}}{2}}  - H_{\frac{{x - y - 2}}{2}} } \right).
\end{align}
\end{lemma}
We require Lemma~\ref{ezputyr} in the proof of Corollary~\ref{aspjw9g}.
\begin{lemma}\label{ezputyr}
If $n$ is an integer, then
\begin{equation}\label{z228axh}
H_{\frac{{2n - 1}}{4}}  - H_{\frac{{2n - 3}}{4}}  =  - \frac{8}{{\left( {2n - 1} \right)\left( {2n - 3} \right)}} + ( - 1)^n \pi  + 4\,\sum_{k = 1}^{n - 2} {\frac{1}{{4k + 3 - 2n}}} .
\end{equation}
\end{lemma}
\begin{proof}
Using $x=a-b$ in Euler's reflection formula for harmonic numbers,
\begin{equation}\label{i3icspw}
H_{ - x}  = H_x  - \frac{1}{x} + \pi \cot \left( {\pi x} \right),\quad x\in\mathbb C\setminus\mathbb Z,
\end{equation}
the shifted harmonic number recurrence formula
\begin{equation*}
H_{b - a}  - H_b  =  - \frac{1}{b} + \sum_{k = 1}^{a - 1} {\frac{1}{{k - b}}} 
\end{equation*}
can be written as
\begin{equation}\label{jc6x33s}
H_{a - b}  - H_b  = \frac{1}{{a - b}} - \frac{1}{b} - \pi \cot \left( {\pi \left( {a - b} \right)} \right) + \sum_{k = 1}^{a - 1} {\frac{1}{{k - b}}} .
\end{equation}
Choosing $a=n-1$ and $b=(2n-3)/4$ in~\eqref{jc6x33s} gives~\eqref{z228axh}.
\end{proof}
\begin{theorem}
If $r$ is a complex number that is not a non-positive integer, then
\begin{equation}\label{inr99ht}
\sum_{k = 0}^\infty  {\frac{{\left( {\frac{1}{2}} \right)_k^2 }}{{k!^2 }}\,\frac{{\binom{{2k + 2r}}{{2r}}}}{{\binom{{k + r}}{r}^2 }}\,\frac{1}{{2k + r}}}  = \frac{1}{{2r}}\,\frac{{r!}}{{\left( {\frac{1}{2}} \right)_r }}\,\left( {H_{\frac{{2r - 1}}{4}}  - H_{\frac{{2r - 3}}{4}}  + \pi } \right)\,\frac1\pi.
\end{equation}
\end{theorem}
\begin{proof}
Set $x=(r+1)/2$ and $y=r/2$ in~\eqref{iuxoyr8} to obtain
\begin{equation}
\sum_{k = 0}^\infty  {\frac{{\left( {\frac{1}{2} + r} \right)_k \left( {\frac{1}{2}} \right)_k }}{{k!\left( {1 + r} \right)_k }}\frac{1}{{2k + r}}}  = \frac{1}{{2r}}\,\frac{{r!}}{{\left( {\frac{1}{2}} \right)_r }}\,\left( {H_{\frac{{2r - 1}}{4}}  + H_{ - \frac{1}{4}}  - H_{\frac{{2r - 3}}{4}}  - H_{ - \frac{3}{4}} } \right)\,\frac{1}{\pi }
\end{equation}
Now use~\eqref{bjv8z0f} and~\eqref{nly5agg} and note from~\eqref{i3icspw} that
\begin{equation*}
H_{ - \frac{1}{4}}  - H_{ - \frac{3}{4}}  = \pi .
\end{equation*}
\end{proof}
\begin{corollary}\label{aspjw9g}
If $n$ is a positive integer, then
\begin{align*}
&\sum_{k = 0}^\infty  {\frac{{\left( {\frac{1}{2}} \right)_k^2 }}{{k!^2 }}\,\frac{{\binom{{2k + 2n}}{{2n}}}}{{\binom{{k + n}}{n}^2 }}\,\frac{1}{{2k + n}}}\\
&\qquad  = \frac{1}{{2n}}\frac{{n!}}{{\left( {\frac{1}{2}} \right)_n }}\left( { - \frac{8}{{\left( {2n - 1} \right)\left( {2n - 3} \right)}} + \left( {( - 1)^n  + 1} \right)\pi  + 4\sum_{k = 1}^{n - 2} {\frac{1}{{4k + 3 - 2n}}} } \right)\,\frac1\pi.
\end{align*}
In particular,
\begin{align}
\sum_{k = 0}^\infty  {\frac{{\left( {\frac{1}{2}} \right)_k^2 }}{{k!^2 \left( {k + 1} \right)}}}  &= \frac{4}{\pi }\text{ or }\sum_{k = 0}^\infty  {\frac{{\binom{{2k}}{k}C_k }}{{2^{4k} }}}=\frac4\pi, \\
\sum_{k = 0}^\infty  {\frac{{C_k^2 }}{{2^{4k} }}\,\frac{{\left( {2k + 1} \right)\left( {2k + 3} \right)}}{{\left( {k + 2} \right)}}}  &=  - \frac{{16}}{{3\pi }} + 4,
\end{align}
and
\begin{equation}
\sum_{k = 0}^\infty  {\frac{{\left( {\frac{1}{2}} \right)_k^2 }}{{k!^2 }}\,\frac{{\left( {2k + 1} \right)\left( {2k + 5} \right)}}{{\left( {k + 1} \right)\left( {k + 2} \right)\left( {k + 3} \right)}}}  = \frac{{208}}{{45\pi }}.
\end{equation}
\end{corollary}
\begin{proof}
Put $r=n$ in~\eqref{inr99ht} and apply Lemma~\ref{ezputyr}.
\end{proof}

Our final example of a hypergeometric series from which we derive families of Ramanujan-like series for $1/\pi$ is an identity found recently by Campbell~\cite{campbell26} and stated here in lemma~\ref{q52sf7y}.
\begin{lemma}\label{q52sf7y}
If $a$, $b$, and $n$ are suitably bounded complex numbers, then
\begin{equation}\label{ta9kbu1}
\sum_{k = 0}^\infty  {( - 1)^k\, \frac{{\left( a \right)_k \left( { - n} \right)_k \left( {1 + a - b} \right)_k }}{{k!\left( b \right)_k \left( {1 + a + n} \right)_k }}\,\left( {2k + a} \right)}  = \frac{{a\,\left( {1 + a} \right)_n }}{{\left( b \right)_n }}.
\end{equation}
\end{lemma} 
\begin{theorem}
If $p$ and $q$ are complex numbers such that $q-p$ is not a negative integer and $p$ and $q$ are not negative integers, then
\begin{equation}\label{odtv7r1}
\sum_{k = 0}^\infty  {( - 1)^k\, \frac{{\left( {\frac{1}{2}} \right)_k^3 }}{{k!^3 }}\,\frac{{4k + 1 - 2p}}{{\binom{{k - 1/2}}{p}^2 \binom{{k - 1/2}}{q}\binom{{q - p + k}}{k}}}}  = \frac{{\left( {q - p} \right)!p!^2 q!}}{{\left( {\frac{1}{2}} \right)_p \left( {\frac{1}{2}} \right)_q^2 }}\,\frac{1}{{\cos \left( {\pi p} \right)\cos \left( {\pi q} \right)}}\,\frac{2}{\pi }.
\end{equation}
\end{theorem}
\begin{proof}
Evaluate~\eqref{ta9kbu1} at $a=1/2-p$, $b=1$, and $n=q-1/2$ to obtain
\begin{equation}\label{mpfy4qi}
\sum_{k = 0}^\infty  {( - 1)^k\, \frac{{\left( {1/2 - p} \right)_k^2 \left( {1/2 - q} \right)_k }}{{k!^2 \left( {1 + q - p} \right)_k }}\,\left( {4k + 1 - 2p} \right)}  = \left( {1 - 2p} \right)\,\frac{{\left( {3/2 - p} \right)_{q - 1/2} }}{{\left( 1 \right)_{q - 1/2} }}.
\end{equation}
Now use~\eqref{l5fmdri} and ~\eqref{nly5agg} and note from~\eqref{ex_pochhammer},~\eqref{wwymz81}, and~\eqref{zvg11tu} that
\begin{equation}
(1)_{q-\frac12}=\left(\frac12\right)_q\sqrt\pi
\end{equation}
and
\begin{equation}
\left( {3/2 - p} \right)_{q - 1/2}  = \frac{{\left( {q - p} \right)!}}{{\left( {1/2 - p} \right)}}\,\frac{{\left( {\frac{1}{2}} \right)_p \cos \left( {\pi p} \right)}}{{\sqrt \pi  }},
\end{equation}
so that
\begin{equation}\label{q9lzxo2}
\left( {1 - 2p} \right)\,\frac{{\left( {3/2 - p} \right)_{q - 1/2} }}{{\left( 1 \right)_{q - 1/2} }} = \left( {q - p} \right)!\cos \left( {\pi p} \right)\,\frac{{\left( {\frac{1}{2}} \right)_p }}{{\left( {\frac{1}{2}} \right)_q }}\,\frac{2}{\pi }.
\end{equation}
\end{proof}
\begin{remark}
Equation~\eqref{odtv7r1} is the $r=p$ limit of~\eqref{czrzlst}
since
\begin{align*}
\lim_{r\to p}\,\frac{{r - p}}{{\left( {r - p} \right)_{q + 1/2} }}&= \lim_{r\to p}\,\frac{{\left( {r - p} \right)\Gamma \left( {r - p} \right)}}{{\Gamma \left( {r - p + q + 1/2} \right)}}\\
& = \lim_{r\to p}\,\frac{{\Gamma \left( {r - p + 1} \right)}}{{\Gamma \left( {r - p + q + 1/2} \right)}}\\
& = \frac{1}{{\Gamma \left( {q + 1/2} \right)}}\\
&=\frac1{\left(\frac12\right)_q\sqrt\pi}.
\end{align*}
\end{remark}
\begin{corollary}
If $m$ and $n$ are non-negative integers, then
\begin{align}
&\sum_{k = 0}^\infty  {( - 1)^k\, \frac{{\left( {\frac{1}{2}} \right)_k^3 }}{{k!^3 }}\,\frac{{4k + 1 - 2m}}{{\prod_{j = 1}^m {\left( {2k - 2j + 1} \right)^2 } \prod_{j = 1}^n {\left( {2k - 2j + 1} \right)} \prod_{j = 1}^{n - m} {\left( {k + j} \right)} }}}\nonumber\\
&\qquad  = \frac{{( - 1)^{m + n} }}{{\left( {\frac{1}{2}} \right)_m \,\left( {\frac{1}{2}} \right)_n^2 }}\,\frac{1}{{2^{2m + n - 1} }}\frac{1}{\pi }.
\end{align}
\end{corollary}
\begin{proof}
Set $p=m$ and $q=n$ in~\eqref{mpfy4qi} and use~\eqref{mq7ny8m} and~\eqref{q9lzxo2}.
\end{proof}
\begin{theorem}
If $p$ and $q$ are complex numbers such that $q-p$ is not a negative integer and $p$ and $q$ are not negative integers, then
\begin{align}\label{uzeaxl4}
&\sum_{k = 0}^\infty  {( - 1)^k\, \frac{{\left( {\frac{1}{2}} \right)_k^3 }}{{k!^3 }}\,\frac{{\left(4k + 1 - 2p\right)\left(H_{-1/2-p+k}-H_{-1/2-p}+H_k\right)}}{{\binom{{k - 1/2}}{p}^2 \binom{{k - 1/2}}{q}\binom{{q - p + k}}{k}}}}\nonumber\\
&\qquad  = \frac{{\left( {q - p} \right)!p!^2 q!}}{{\left( {\frac{1}{2}} \right)_p \left( {\frac{1}{2}} \right)_q^2 }}\,\frac{H_{q-1/2}}{{\cos \left( {\pi p} \right)\cos \left( {\pi q} \right)}}\,\frac{2}{\pi }.
\end{align}
\end{theorem}
\begin{proof}
Differentiate~\eqref{ta9kbu1} with respect to $b$ to obtain
\begin{align}\label{gz9u2ss}
&\sum_{k = 0}^\infty  {( - 1)^k\, \frac{{\left( a \right)_k \left( { - n} \right)_k \left( {1 + a - b} \right)_k }\left(H_{a-b+k}-H_{a-b}+H_{b+k-1}-H_{b-1}\right)}{{k!\left( b \right)_k \left( {1 + a + n} \right)_k }}\,\left( {2k + a} \right)}\nonumber\\
&\qquad  = \frac{{a\,\left( {1 + a} \right)_n }}{{\left( b \right)_n }}\,\left(H_{b+n-1}-H_{b-1}\right).
\end{align}
Now set $a=1/2-p$, $b=1$, and $n=q-1/2$ to get
\begin{align}\label{x0kmq9z}
&\qquad\sum_{k = 0}^\infty  {( - 1)^k\, \frac{{\left( {1/2 - p} \right)_k^2 \left( {1/2 - q} \right)_k }}{{k!^2 \left( {1 + q - p} \right)_k }}\,\left( {4k + 1 - 2p} \right)\left(H_{-1/2-p+k}-H_{-1/2-p}+H_k\right)}\nonumber\\
&\qquad  = \left( {1 - 2p} \right)\,\frac{{\left( {3/2 - p} \right)_{q - 1/2} }}{{\left( 1 \right)_{q - 1/2} }}\,H_{q-1/2}.
\end{align}
Use of~\eqref{l5fmdri},~\eqref{nly5agg}, and~\eqref{q9lzxo2} completes the proof.
\end{proof}
\begin{corollary}
If $m$ and $n$ are non-negative integers, then
\begin{align}
&\sum_{k = 0}^\infty  {( - 1)^k\, \frac{{\left( {\frac{1}{2}} \right)_k^3 }}{{k!^3 }}\,\frac{{\left(4k + 1 - 2m\right)}\left(2O_{k-m}-2O_m+H_k\right)}{{\prod_{j = 1}^m {\left( {2k - 2j + 1} \right)^2 } \prod_{j = 1}^n {\left( {2k - 2j + 1} \right)} \prod_{j = 1}^{n - m} {\left( {k + j} \right)} }}}\nonumber\\
&\qquad  = \frac{{( - 1)^{m + n} }}{{\left( {\frac{1}{2}} \right)_m \,\left( {\frac{1}{2}} \right)_n^2 }}\,\frac{1}{{2^{2m + n - 2} }}\frac{1}{\pi }\,\left(O_n-\ln 2\right).
\end{align}
In particular,
\begin{equation}
\sum_{k = 0}^\infty  {( - 1)^k \frac{{\left( {\frac{1}{2}} \right)_k^3 }}{{k!^3 }}\,\frac{{\left( {4k + 1} \right)H_{2k} }}{{\prod_{j = 1}^n {\left( {2k - 2j + 1} \right)\left( {k + j} \right)} }}}  = \frac{{( - 1)^n }}{{2^{n - 1} }}\,\frac{{\left( {O_n  - \ln 2} \right)}}{{\left( {\frac{1}{2}} \right)_n^2 }}\,\frac{1}{\pi },
\end{equation}
with the special value
\begin{equation*}
\sum_{k = 0}^\infty  {( - 1)^k \frac{{\left( {\frac{1}{2}} \right)_k^3 }}{{k!^3 }}\,\left( {4k + 1} \right)H_{2k} }  =  - \frac{{2\ln 2}}{\pi }.
\end{equation*}
\end{corollary}
\begin{proof}
Set $p=m$ and $q=n$ in~\eqref{x0kmq9z} and use~\eqref{nerd5x1},~\eqref{mq7ny8m}, and~\eqref{q9lzxo2}.
\end{proof}


\begin{thebibliography}{99}

\bibitem{bailey35} W. N. Bailey, \emph{Generalized Hypergeometric Series}, Cambridge University Press, 1935.

\bibitem{bauer} G. Bauer, Von den Coefficienten der Reihen von Kugelfunktionen einer Variabeln, \emph{Journal f\"ur die Reine und Angewandte Mathematik} {\bf 56} (1859), 101--121.

\bibitem{campbell24} J. M. Campbell and P. Levrie, The Bauer-Ramanujan formula: historical analyses and perspectives, \emph{British Journal for the History of Mathematics} {\bf 39}:3 (2024), 193--214.

\bibitem{campbell26} J. M. Campbell, Applications of a class of transformations of complex sequences, \emph{Journal of the Ramanujan Mathematical Society}, to appear .

\bibitem{chu14} W. Chu and W. Zhang, Accelerating Dougall's  ${}_5F_4$-sum and infinite series involving $\pi$, \emph{Mathematics of Computation} {\bf 285} (2014), 475--512.

\bibitem{dougall06} J. Dougall, Vandermonde's theorem and some more general expansions, \emph{Proceedings of the Edinburgh Mathematical Society} {\bf 25} (1906), 114--152.

\bibitem{ekhad94} S. B. Ekhad and D. Zeilberger, A $WZ$ proof of Ramanujan's formula for $\pi$, in J. M. Rassias, ed., \emph{Geometry, Analysis and Mechanics}, World Scientific, 1994, pp. 107--108.

\bibitem{hou24} Q. Hou, H. He and X. Wang, Ramanujan-inspired series for $1/\pi$ involving harmonic numbers, \emph{Journal of Differential Equations of Applications} {\bf 30}:6 (2024), 775--786.

\bibitem{lavoie66} J. V. Lavoie, Partial sums of coefficients of well-poised hypergeometric series, \emph{Bolletino dell'Unione Matematica Italiana, Serie 3,} {\bf 21}:4 (1966), 346--352.

\bibitem{levrie10} P. Levrie, Using Fourier-Legendre expansions to derive series for $1/\pi$ and $1/\pi^2$, \emph{Ramanujan Journal} {\bf 22} (2010), 221--230.

\bibitem{ramanujan14} S. Ramanujan, Modular equations and approximations for $\pi$, \emph{Quarterly Journal of Pure and Applied Mathematics} {\bf 45} (1914), 350--372.

\end{thebibliography}
\end{document}